\documentclass{amsart}
\usepackage{amsmath,amsfonts,mathrsfs,amssymb,enumerate,graphicx,subfig,hyperref}
\usepackage{color} 
\newtheorem{theorem}{Theorem}

\newtheorem{proposition}{Proposition}[section]

\newtheorem{lemma}[proposition]{Lemma}

\DeclareMathOperator{\dimaff}{dim_{\mathsf{aff}}}

\DeclareMathOperator{\GL}{GL}

\DeclareMathOperator{\Span}{span}
\newcommand{\threebar}[1]{{\left\vert\kern-0.25ex\left\vert\kern-0.25ex\left\vert #1 
    \right\vert\kern-0.25ex\right\vert\kern-0.25ex\right\vert}}

\newcommand{\R}{\mathbb{R}}

\newcommand{\C}{\mathbb{C}}

\title[On affine IFS with invariant affine subspaces] {On affine iterated function systems which robustly admit an invariant affine subspace}

\author{Ian D. Morris}
\address{School of Mathematical Sciences, Queen Mary, University of London, Mile End Road, London E1 4NS, U.K.}
\email{i.morris@qmul.ac.uk }

\begin{document}

\begin{abstract}
In this note we give a simple sufficient condition for an affine iterated function system to admit an invariant affine subspace persistently with respect to changes in the translation parameters. This yields further examples of tuples of contracting linear maps which do not satisfy the conclusions of  Falconer's theorem on the Hausdorff dimension of almost every self-affine set. We also obtain new examples of iterated function systems of similarity transformations which cannot satisfy the open set condition for any choice of translation parameters, and resolve a related question of Peres and Solomyak.
\end{abstract}
\maketitle




\section{Introduction}

\subsection{Motivation and background} 
We recall that an \emph{iterated function system} is a tuple of contracting transformations $T_1,\ldots,T_N$ of a complete metric space $X$. In this note $X$ will always be either $\R^d$ or one of its affine subspaces and the metric on $X$ will be that induced by a norm (but not necessarily the Euclidean norm). It is classical that for every iterated function system $(T_1,\ldots,T_N)$ acting on such a space $X$ there exists a unique nonempty compact set $Z\subseteq X$ which satisfies $Z=\bigcup_{i=1}^N T_iZ$ and this set is usually called the \emph{attractor} or \emph{limit set} of the iterated function system. 

We recall that $(T_1,\ldots,T_N)$ is said to satisfy the \emph{open set condition} if there exists a nonempty open set $U \subseteq X$ such that $\bigcup_{i=1}^N T_iU \subseteq U$ and such that the images $T_iU$ are pairwise disjoint. In the case where every $T_i$ is a similarity transformation of $\R^d$ and the open set condition is satisfied, the dimension characteristics of the attractor have been well-understood since the foundational work of J.E. Hutchinson \cite{Hu81} in 1981. When the open set condition fails, or when the transformations are allowed to be arbitrary affine contractions as opposed to similitudes, the dimension theory of the attractor is much less well understood and the relaxation of both of these conditions is a substantial topic of ongoing research (see for example \cite{BaHoRa19,Fe19,Ho20,HoRa19}). An early landmark result in this area is the following result (see \cite{Fa88}):
\begin{theorem}[Falconer]\label{th:falc}
Let $N,d \geq 1$ and $A_1,\ldots,A_N \in \GL_d(\R)$ and suppose that $\max_{1 \leq i \leq N}\|A_i\|<\frac{1}{3}$. Then for Lebesgue almost every $(v_1,\ldots,v_N) \in \R^d\oplus \cdots \oplus \R^d \simeq \R^{dN}$ the Hausdorff dimension of the attractor of the affine iterated function system $T_1,\ldots,T_N \colon \R^d \to \R^d$ defined by $T_ix:=A_ix+v_i$ is equal to either the \emph{affinity dimension} of $(T_1,\ldots,T_N)$ or to $d$, whichever is smaller.
\end{theorem}
Here the affinity dimension (in some works called the singularity dimension or nominal dimension) is a real number depending only on $(A_1,\ldots,A_N)$ which is defined as follows. Let $M_d(\R)$ denote the vector space of all $d \times d$ real matrices and recall that for $A \in M_d(\R)$ the \emph{singular values} of $A$ are defined to be the eigenvalues of the positive semidefinite matrix $A^TA$, listed in decreasing order with repetition according to multiplicity. We denote the singular values of $A$ by $\sigma_1(A) \geq \sigma_2(A)\geq \cdots \geq \sigma_d(A)$. Following \cite{Fa88}, for every $s \geq 0$ we define a function $\varphi^s \colon M_d(\R) \to [0,\infty)$ by
\[\varphi^s(A):=\left\{\begin{array}{cl}\sigma_1(A)\cdots \sigma_{\lfloor s\rfloor}(A) \sigma_{\lceil s\rceil}(A)^{s-\lfloor s\rfloor}&\text{if }0 \leq s \leq d,\\
|\det A|^{\frac{s}{d}}&\text{otherwise.}\end{array}\right.\]
The \emph{affinity dimension} of $(A_1,\ldots,A_N) \in M_d(\R)^N$ is then defined by
\[\dimaff(A_1,\ldots,A_N) :=\inf\left\{s \geq 0\colon \lim_{n \to \infty} \frac{1}{n}\log \sum_{i_1,\ldots,i_n=1}^N \varphi^s(A_{i_1}\cdots A_{i_n}) \leq 0\right\}.\](When $A_1,\ldots,A_N$ are invertible this infimum is attained, a fact which will be appealed to implicitly in \S\ref{ss:e2} below.) When  $T_1,\ldots,T_N\colon \R^d \to \R^d$ are affine transformations of the form $T_ix \equiv A_ix+v_i$ we will also write $\dimaff(T_1,\ldots,T_N):=\dimaff(A_1,\ldots,A_N)$. In the case where the transformations $T_i$ are similarities with contraction ratios given by $r_i \in (0,1)$, we have $\varphi^s(A_{i_1}\cdots A_{i_n})\equiv r_{i_1}^s\cdots r_{i_n}^s$ for all $s \geq 0$ and the affinity dimension is simply the unique solution $s\geq 0$ to the equation $\sum_{i=1}^N r_i^s=1$. This fact will also be appealed to in this introduction and in \S\ref{se:ex} below.
 
 While Theorem \ref{th:falc} is conventionally stated using the Euclidean norm, with only minor modifications to the proof any norm on $\R^d$ may be used. (This modification has the advantage that both the conclusion of the theorem and its hypotheses become invariant with respect to change of co-ordinates in $\R^d$, whereas in the standard formulation only the conclusion has this property.) Subsequent work of B. Solomyak \cite{So98} demonstrated that the constant $\frac{1}{3}$ in Theorem \ref{th:falc} may be relaxed to a value of $\frac{1}{2}$ and also showed that no further improvement in the value of this constant is possible in general. In this note we will be concerned with examples of tuples of linear maps $(A_1,\ldots,A_N)$ which are not $\frac{1}{2}$-contracting with respect to any norm on $\R^d$ and for which the conclusions of Theorem \ref{th:falc} do not hold. We first review some prior examples of this phenomenon.


%
%

\subsection{The example of Przytycki, Urba\'nski, Edgar and Solomyak.} An example based on a construction of Przytycki and Urba\'nski \cite{PrUr89} can be applied to show that the constant in Theorem \ref{th:falc} cannot be replaced with $\frac{1}{2}+\varepsilon$ for any $\varepsilon>0$; the relevance of Przytycki and Urba\'nski's construction to this problem was first remarked upon by G.A. Edgar \cite{Ed92} and expanded upon by B. Solomyak \cite{So98}. The example is as follows. Define $A_1, A_2 \in \GL_2(\R)$ by
\[A_1=A_2=\begin{pmatrix}1/\beta&0\\ 0&\frac{1}{2}\end{pmatrix}\]
where $\beta \in (1,2)$ is a Pisot-Vijayaraghavan number, that is, a real algebraic integer greater than $1$ all of whose Galois conjugates are less than one in absolute value. Edgar noted in \cite{Ed92} that the choice of translation vectors $v_1, v_2 \in \R^2$ has very little effect on the dimension of the attractor, as follows. If $T_1$, $T_2$ are given by $T_ix:=A_ix +v_i$ for some $v_1, v_2 \in \R^2$, then changing co-ordinates via a translation we may without loss of generality assume that $v_1$ is the zero vector. On the other hand, as long as $v_2$ does not lie on the horizontal or vertical axis in the translated co-ordinates (or equivalently, if $v_1$ and $v_2$ did not lie on the same horizontal or vertical line in the original co-ordinates) then a further, linear change of co-ordinates by conjugation with a diagonal matrix allows us to assume without loss of generality that $v_2=(1-1/\beta,1/2)$ so that the fixed points of $T_1$ and $T_2$ are $(0,0)$ and $(1,1)$ respectively. It follows that for Lebesgue almost every choice of $v_1, v_2 \in \R^2$ the dimension of the attractor equals that of the iterated function system defined by
\[T_1 \begin{pmatrix} x\\ y\end{pmatrix}:=\begin{pmatrix}1/\beta&0\\ 0&\frac{1}{2}\end{pmatrix}\begin{pmatrix}x\\y\end{pmatrix},\qquad T_2 \begin{pmatrix} x\\ y\end{pmatrix}:=\begin{pmatrix}1/\beta&0\\ 0&\frac{1}{2}\end{pmatrix}\begin{pmatrix}x\\y\end{pmatrix} +\begin{pmatrix} 1-\frac{1}{\beta} \\ \frac{1}{2}\end{pmatrix}\]
which had earlier  been studied in \cite{PrUr89}.
This iterated function system satisfies the open set condition with $U:=(0,1) \times (0,1)$. Let $\Pi \colon \R^2 \to \R$ denote orthogonal projection on the first co-ordinate. The Pisot-Vijayaraghavan property of $\beta$ implies that for every $n \geq 1$ the number of \emph{distinct} maps of the form $\Pi T_{i_1}\cdots T_{i_n}$ is bounded by a constant times $\beta^n$. Since $\beta<2$ this implies that significant numbers of the distinct images $T_{i_1}\cdots T_{i_n}U$ are arranged close to one another in vertical columns, and this implies the existence of a sequence of covers for the attractor giving a stronger upper bound for the Hausdorff dimension than that arising from Theorem \ref{th:falc}. For a rigorous argument we direct the reader to \cite{PrUr89}. It had been observed by Edgar that the preceding example implies that the constant $\frac{1}{3}$ in Theorem \ref{th:falc} cannot be specifically improved to the number $\frac{\sqrt{5}-1}{2}\simeq 0.618\ldots$ (which is the reciprocal of the golden ratio, a Pisot-Vijayaraghavan number) and Solomyak subsequently noted in \cite{So98} that since Pisot-Vijayaraghavan numbers accumulate from below at $2$, this class of examples precludes the replacement of the contraction constant in Theorem \ref{th:falc} with any number strictly greater than one half.


\subsection{An example of Simon and Solomyak.}\label{ss:ss} The preceding example operates by using an algebraic property of the horizontal contraction ratio common to the two maps to force two distinct compositions $T_{i_1}\cdots T_{i_n}$ and $T_{j_1}\cdots T_{j_n}$ to coincide when composed with a projection. An alternative construction introduced by K. Simon and B. Solomyak in \cite{SiSo02} operates more directly by inducing coincidences among the maps $T_{i_1}\cdots T_{i_n}$ themselves. We present their example in a slightly modified form. Let $d \geq 2$ and let $T_1, T_2 \colon \R^d \to \R^d$ be given by $T_ix:=\lambda x+v_i$ where $\lambda \in (0,1)$ satisfies an equation of the form $\sum_{r=1}^n \lambda^{r-1}(j_r-k_r) =0$ for some $n \geq 1$ and some $j_1,\ldots,j_n, k_1,\ldots,k_n \in \{1,2\}$ such that $(j_1,\ldots,j_n) \neq (k_1,\ldots,k_n)$, all of which we fix for the remainder of this discussion. It is shown in \cite{SiSo02} that such values of $\lambda$ occur densely in the interval $(\frac{1}{2},1)$. We suppose additionally that $2\lambda^d<1$ which implies that the affinity dimension of $(T_1,T_2)$ is strictly less than $d$.  For every $x \in \R^d$ we have
\begin{align*}T_{j_1}\cdots T_{j_n}x - T_{k_1}\cdots T_{k_n}x&= \left(\lambda^nx+\sum_{r=1}^n \lambda^{r-1} v_{j_r}\right) - \left(\lambda^nx+\sum_{r=1}^n \lambda^{r-1} v_{k_r}\right)\\
&=\sum_{r=1}^n \lambda^{r-1} (v_{j_r} - v_{k_r})\\
&=\sum_{r=1}^n \lambda^{r-1}(j_r-k_r)(v_2-v_1)=0\end{align*}
and this yields $T_{j_1}\cdots T_{j_n}=T_{k_1}\cdots T_{k_n}$ irrespective of the precise values of $v_1$ and $v_2$. Now let $T_1',\ldots,T_m'$ be a complete list of the \emph{distinct} maps of the form $T_{i_1} \cdots T_{i_n}$, where by the preceding observation we have $m<2^n$. Let $Z$ denote the attractor of $(T_1,T_2)$. By iterating the relation $Z=\bigcup_{i=1}^2 T_iZ$ we obtain
\[Z=\bigcup_{i=1}^2 T_iZ = \bigcup_{i_1,\ldots,i_n=1}^2 T_{i_1}\cdots T_{i_n}Z=\bigcup_{i=1}^{m}T_i'Z\]
and therefore $Z$ is the attractor of the iterated function system $(T_1',\ldots,T_m')$ as well as of $(T_1,T_2)$. It is straightforward to check using the formula $\sum_{i=1}^N r_i^s=1$ that if $s$ denotes the affinity dimension of $(T_1,T_2)$ and $s'$ denotes that of $(T_1',\ldots,T_m')$ then $2^n\lambda^{ns'}>m\lambda^{ns'}=1=2\lambda^s=2^n\lambda^{ns}$ and this implies $s'<s$. By a further result of Falconer (see \cite{Fa88}) the affinity dimension of an affine iterated function system is unconditionally an upper bound for the Hausdorff dimension of its attractor, so the Hausdorff dimension of $Z$ is bounded by $\dimaff(T_1',\ldots,T_m')$ which is strictly smaller than $\dimaff(T_1,T_2)$. Since the precise value of the vectors $v_1,v_2 \in \R^d$ played no role in this argument, and since $\dimaff(T_1,T_2)<d$, this construction yields further examples of tuples of linear maps for which the conclusions of Theorem \ref{th:falc} do not hold.

This example $(T_1,T_2)$ may easily be extended to any strictly larger affine iterated function system $(T_1,T_2,\ldots,T_N)$ which includes the two maps $T_1$ and $T_2$ defined above and which satisfies $\dimaff(T_1,\ldots,T_N)<d$, as long as the additional maps $T_i$ are chosen to be invertible. Indeed, we may apply precisely the same arguments to $T_1$ and $T_2$ to deduce that for some $n \geq 1$ the attractor of $(T_1,\ldots,T_N)$ is also the attractor of an iterated function system formed from a proper subset of the $N^n$ maps $T_{i_1}\cdots T_{i_n}$, and by a theorem of J. Bochi and the author (see \cite{BoMo18}) the latter iterated function system has smaller affinity dimension than the former, which implies that the Hausdorff dimension of the attractor is strictly less than $\dimaff(T_1,\ldots,T_N)$. With appropriate choices of additional maps $T_3,\ldots,T_N$ this construction yields examples of affine iterated function systems which are \emph{irreducible} (that is, whose linearisations do not preserve a nonzero proper subspace of $\R^d$) but for which the conclusions of Theorem \ref{th:falc} do not hold. This answers negatively a question of Peres and Solomyak \cite[Question 3.4]{PeSo98}. Indeed, by this mechanism even stronger algebraic properties such as strong irreducibility or the Zariski density of the semigroup generated by the linearisations may be obtained with equal ease.

We note one further variation on this example which gives an appealingly concise answer to Peres and Solomyak's question. For this construction it is convenient to identify $\R^2$ with $\C$. If $\lambda \in \C\setminus\R$ satisfies $|\lambda|<1/\sqrt{2}$ and solves an equation of the form $\sum_{r=1}^n \lambda^{r-1}(j_r-k_r) =0$ as before, define two similarity transformations $T_1, T_2 \colon \C \to \C$ by $T_iz:=\lambda z +w_i$ where $w_1, w_2 \in \C$ are arbitrary. By exactly the same calculations as before we have $T_{j_1}\cdots T_{j_n}=T_{k_1}\cdots T_{k_n}$ and the Hausdorff dimension of the attractor of $(T_1,T_2)$ is smaller than the affinity dimension which in turn is smaller than $2$. If $\lambda/|\lambda|$ is not a root of unity then this example is moreover \emph{strongly irreducible:} its linearisation does not preserve any finite union of lines through the origin. In simple cases such as $\lambda^4+\lambda^3+\lambda^2-\lambda+1=0$ this property may be verified directly by computing the minimal polynomial of $\lambda^2/|\lambda|^2=\lambda/\overline{\lambda}$ and checking that it is not cyclotomic.

\subsection{A simpler example.}\label{ss:simp} The following rather different example is much simpler to formulate and to prove than the previous two. Although it has circulated informally in seminars and internet forums\footnote{The author learned of this example from a seminar given by T.M. Jordan in 2018. An answer posted on the mathematics research forum MathOverflow by P. Shmerkin  in 2013  appeals to this example: {see \href{http://mathoverflow.net/questions/136193}{\tt{http://mathoverflow.net/questions/136193}}}.} the author is unaware of any reference in the formal literature. Let $\lambda \in (\frac{1}{2},1)$ be arbitrary, define $A_i:=\lambda I \in \GL_2(\R)$ for $i=1,2$, and for each 
pair $v_1, v_2 \in \R^2$ consider the affine iterated function system $T_1,T_2 \colon \R^2 \to \R^2$ defined by $T_ix:=A_ix+v_i$. As long as $v_1$ and $v_2$ are distinct, we may apply an affine co-ordinate change so as to assume without loss of generality that $T_1$ has its unique fixed point located at the origin whereas $T_2$ has its unique fixed point located at $(1,0)$. The attractor of $T_1$ and $T_2$ in the new co-ordinates is therefore just the unit interval in $\R^2$ and has Hausdorff dimension equal to $1$. On the other hand the affinity dimension of $(T_1,T_2)$ is easily computed to equal $\log 2 / \log(1/\lambda)>1$. By appeal to Hutchinson's theorem it also follows that the pair of maps $(T_1,T_2)$ does not satisfy the open set condition for any choice of the parameters $v_1, v_2$.  

\subsection{Main results}

 In this note we develop the third of the three examples considered above into a much more general class of examples based on the mechanism of choosing linear maps $A_1,\ldots,A_N$ which force the existence of an invariant affine subspace of $\R^d$ for all choices of the translation parameters $v_1,\ldots,v_N$. We prove the following result:
\begin{theorem}\label{th:zeroth-main}
Let $T_1,\ldots,T_N$ be affine transformations of $\R^d$ all of which are contracting with respect to some fixed norm on $\R^d$, and for each $i=1,\ldots,N$ let $A_i \in M_d(\R)$ denote the linearisation of $T_i$.
Let $D \geq 1$ be the dimension of the smallest subalgebra of $M_d(\R)$ which contains $A_1,\ldots,A_N$ and the identity. Then there exists an affine subspace $X \subseteq \R^d$ with dimension not greater than $(N-1)D$ such that $T_iX \subseteq X$ for every $i=1,\ldots,N$. The affine subspace $X$ contains the attractor of $(T_1,\ldots,T_N)$ as a subset; in particular if $(N-1)D < \dimaff (T_1,\ldots,T_N)$ then the Hausdorff dimension, of the attractor of $(T_1,\ldots,T_N)$ is strictly smaller than the affinity dimension.
\end{theorem}
Since the above criterion depends only on the linearisations of the affine transformations $T_i$ and not on their translation components, the conclusion of Theorem \ref{th:zeroth-main} holds for all choices of translation vectors $v_i$ when the linearisations $A_i$ remain fixed. We also remark that the only property of Hausdorff dimension needed in the proof is that every subset of an $\ell$-dimensional affine space has Hausdorff dimension bounded above by $\ell$; in particular the conclusion of Theorem \ref{th:zeroth-main} remains valid if the Assouad dimension, packing dimension or upper of lower box dimension is substituted for the Hausdorff dimension. It is not clear to the author whether or not the bound $(N-1)D$ in Theorem \ref{th:zeroth-main} is sharp, and we leave the sharpness of this bound as a question for future researchers.

We will obtain Theorem \ref{th:zeroth-main} as a corollary of the following result which describes more precisely the class of affine iterated function systems  which admit an invariant affine subspace for all choices of translation parameters. We recall that a \emph{subvariety} of $\R^d$ is defined to be any set which is equal to the intersection of the zero loci of a collection of polynomial functions $\R^d\to \R$.
\begin{theorem}\label{th:first-main}
Let $N, \ell, d$ be integers satisfying $1 \leq N-1 \leq \ell <d$ and suppose that  $A_1,\ldots,A_N \in M_d(\R)$ are linear maps all of which are contracting with respect to some fixed norm on $\R^d$. Let $\mathcal{Z}_\ell$ denote the set of all $(v_1,\ldots,v_N) \in \bigoplus_{i=1}^N \R^d\simeq \R^{dN}$ with the property that the affine maps $T_1,\ldots, T_N \colon \R^d \to \R^d$ defined by $T_ix:=A_ix+v_i$ preserve an affine subspace $X\subseteq \R^d$ of dimension less than or equal to $\ell$.
Then:
\begin{enumerate}[(i)]
\item
The set $\mathcal{Z}_\ell$ is a subvariety of $\bigoplus_{i=1}^N \R^d$, and in particular either $\mathcal{Z}_\ell=\bigoplus_{i=1}^N \R^d$ or $\mathcal{Z}_\ell$ has zero Lebesgue measure.
\item
We have $\mathcal{Z}_\ell=\bigoplus_{i=1}^N \R^d$ if and only if the following property holds: for every $(N-1)$-dimensional vector subspace $U$ of $\R^d$ there exists a vector subspace $W$ of $\R^d$ which has dimension at most $\ell$, satisfies $U\subseteq W$, and satisfies $AW\subseteq W$ for every $A$ in the subalgebra of $M_d(\R)$ generated by $A_1,A_2,\ldots,A_N$. 
\end{enumerate}
\end{theorem}
Theorem \ref{th:first-main} implies that if $(A_1,\ldots,A_N) \in M_d(\R)^N$ is irreducible in the sense defined in \S\ref{ss:ss}, and is contracting with respect to some norm on $\R^d$, then for an open, dense, full-measure set of translation parameters $(v_1,\ldots,v_N) \in \R^{dN}$ the associated affine iterated function system $(T_1,\ldots,T_N)$ does \emph{not} preserve a proper affine subspace of $\R^d$. 


In the following section we present some example applications of Theorem \ref{th:zeroth-main}. The proofs of Theorems \ref{th:zeroth-main} and \ref{th:first-main} are presented in \S\ref{se:pf} below.

%
%

\section{Examples}\label{se:ex}

Theorem \ref{th:zeroth-main} encompasses the example described in \S\ref{ss:simp} since in that case the algebra generated by $A_1, A_2, I$ is simply $\{\alpha I \colon \alpha \in \R\}\subset M_2(\R)$ which has dimension $1$. On the other hand Theorem \ref{th:zeroth-main} is clearly more general, and in this section we present some further examples.


\subsection{Example 1}
 Let $\phi,\psi \in \R$ and suppose that $A_1, A_2 \in M_4(\R)$ are given by
\[A_1=\begin{pmatrix} R_\phi &0 \\ 0&R_{-\phi}\end{pmatrix},\qquad A_2=\begin{pmatrix} R_\psi &0 \\ 0&R_{-\psi}\end{pmatrix}\]
where $R_\theta \in M_2(\R)$ denotes the matrix corresponding to an anticlockwise rotation of $\R^2$ by angle $\theta$. Let $\varepsilon \in (0,\frac{1}{4})$ and $v_1, v_2 \in \R^2$ be arbitrary and define $T_1, T_2 \colon \R^4 \to \R^4$ by $T_ix:=2^{-\frac{1}{2}+\varepsilon}A_ix +v_i$. Then $T_1$ and $T_2$ are contracting with respect to the Euclidean norm and their affinity dimension is precisely $2/(1-2\varepsilon) \in (2,4)$. On the other hand since $A_1,A_2$ and the identity all belong to the two-dimensional subalgebra
\[\left\{ \begin{pmatrix} \alpha &\beta &0 &0\\ -\beta &\alpha&0&0\\
0&0&\alpha&-\beta \\
0&0&\beta &\alpha
\end{pmatrix} \colon \alpha, \beta \in \R\right\} \subset M_4(\R)\]
there exists an affine subspace of $\R^4$ with dimension at most $2$ which is preserved by $T_1$ and $T_2$. Consequently the attractor of $(T_1, T_2)$ has dimension at most $2$ and in particular has dimension less than the affinity dimension. Since $T_1$ and $T_2$ are similitudes it follows from Hutchinson's theorem that $(T_1,T_2)$ cannot satisfy the open set condition.


\subsection{Example 2}\label{ss:e2} Define two matrices $B_1,B_2 \in M_4(\R)$ by
\[B_1=\begin{pmatrix}0&-1&-1&0\\ 1&0&0&-1\\ 1&0&0&-1\\ 0&1&1&0\end{pmatrix},\qquad B_2=\begin{pmatrix}1&0&0&1\\ 0&1&-1&0\\ 0&-1&1&0\\ 1&0&0&1\end{pmatrix}.\]
The reader may verify that $B_1B_2=B_2B_1=0$, $B_1^2=2B_2-4I$ and $B_2^2=2B_2$. It follows from this that the subalgebra of $M_4(\R)$ generated by $B_1, B_2$ and the identity is simply the linear span of $B_1, B_2$ and $I$, 
\[\left\{\begin{pmatrix}
\alpha & -\beta &-\beta & \gamma\\
\beta &\alpha & -\gamma & -\beta\\
\beta &-\gamma &\alpha & -\beta \\
\gamma &\beta &\beta &\alpha\end{pmatrix}
\colon \alpha,\beta,\gamma \in \R\right\}.\]
In particular if $T_1,T_2 \colon \R^4 \to \R^4$ are affine contractions both having the form $T_ix=(\alpha_i I+\beta_i B_1 +\gamma_i B_2)x+v_i$ then they preserve an affine subspace $X\subset \R^4$ of dimension not greater than $3$. The attractor of $(T_1,T_2)$ consequently has dimension not greater than $3$. If $\alpha_1$ and $\alpha_2$ are chosen in $(2^{-1/3},1)$ and the numbers $\beta_i, \gamma_i$ are taken to be zero then the resulting iterated function system is contracting with respect to the Euclidean norm and is easily seen to have affinity dimension exceeding $3$. By continuity of the affinity dimension (see \cite{FeSh14,Mo16}) this situation persists if $\alpha_1, \alpha_2 \in (2^{-1/3},1)$ and the numbers $\beta_i, \gamma_i$ are chosen sufficiently close to zero.


\subsection{Example 3} Let $d \geq 1$ and $A_1,\ldots,A_N \in \GL_d(\R)$ and recall that the \emph{joint spectral radius} and \emph{lower spectral radius} of $(A_1,\ldots,A_N)$ are defined respectively by
\begin{align*}\overline{\varrho}(A_1,\ldots,A_N)&:= \lim_{n \to \infty} \max_{1 \leq i_1,\ldots, i_n \leq N}  \left\|A_{i_1}\cdots A_{i_n}\right\|^{\frac{1}{n}}\\
&=\inf_{n \geq} \max_{1 \leq i_1,\ldots, i_n \leq N}  \left\|A_{i_1}\cdots A_{i_n}\right\|^{\frac{1}{n}}\end{align*}
and
\begin{align*}\underline{\varrho}(A_1,\ldots,A_N)&:= \lim_{n \to \infty} \min_{1 \leq i_1,\ldots, i_n \leq N}  \left\|A_{i_1}\cdots A_{i_n}\right\|^{\frac{1}{n}}\\
&= \inf_{n \geq 1} \min_{1 \leq i_1,\ldots, i_n \leq N}  \left\|A_{i_1}\cdots A_{i_n}\right\|^{\frac{1}{n}}.\end{align*}
(For a proof of the validity of the above formulas see \cite{Ju09}.) Suppose that there exists an integer $k\geq (N-1)d^2$ such that 
\[N^{-\frac{1}{k}} <\underline{\varrho}(A_1,\ldots,A_N) \leq \overline{\varrho}(A_1,\ldots,A_N)<1.\]
Since the joint spectral radius of $A_1,\ldots,A_N$ is less than $1$ there exists a norm $\threebar{\cdot}$ on $\R^d$ such that $\max_{1 \leq i \leq N}\threebar{A_i}<1$ by a classical result of Rota and Strang (\cite{RoSt60}, reprinted in \cite{Ro03}). For every $B \in M_d(\R)$ let $B^{\oplus k} =B \oplus B \oplus \cdots \oplus B \in M_{kd}(\R)$ denote the direct sum of $k$ identical copies of $B$. Let $v_1,\ldots,v_N \in \R^{kd}$ be arbitrary and define affine transformations $T_1,\ldots,T_N \colon \R^{kd} \to \R^{kd}$ by $T_ix:=A_i^{\oplus k}x+v_i$. We observe that these maps are contractions with respect to the norm on $\R^{kd}$ defined by taking the direct sum of $k$ copies of the norm $\threebar{\cdot}$ on $\R^d$. 

The algebra generated by $A_1^{\oplus k},\ldots,A_N^{\oplus k}, I^{\oplus k}$ has the same dimension as that generated by $A_1,\ldots,A_N, I$, which is at most $d^2$ since the latter algebra is a subalgebra of $M_d(\R)$. Consequently $T_1,\ldots,T_N$ preserve an affine subspace $X \subset \R^{kd}$ of dimension not greater than $(N-1)d^2$ and the dimension of the attractor can be no greater than that value. Since
\begin{align*}\lim_{n \to \infty} \left(\sum_{1 \leq i_1,\ldots,i_n \leq N} \varphi^k(A_{i_1}^{\oplus k}\cdots A_{i_n}^{\oplus k}) \right)^{\frac{1}{n}}&=\lim_{n \to \infty} \left(\sum_{1 \leq i_1,\ldots,i_n \leq N} \|A_{i_1}\cdots A_{i_n}\|^k \right)^{\frac{1}{n}}\\
&\geq N\underline{\varrho}(A_1,\ldots,A_N)^k> 1\end{align*}
we have
\[\dimaff (T_1,\ldots,T_N)> k\geq(N-1)d^2\]
so that the dimension of the attractor is strictly less than the affinity dimension. 

%
%

\section{Proof of Theorems \ref{th:zeroth-main} and \ref{th:first-main}}\label{se:pf}

\subsection{Proof of Theorem \ref{th:first-main}}
We prove Theorem \ref{th:first-main} by reducing it to a series of lemmas. The following two lemmas combined yield Theorem \ref{th:first-main}(ii):
\begin{lemma}\label{le:re}
Let $N, \ell, d$ be integers satisfying $1 \leq N-1 \leq \ell <d$, let $A_1,\ldots,A_N \in M_d(\R)$ be linear maps all of which are contracting with respect to some fixed norm on $\R^d$, and let $v_1,\ldots,v_N \in \R^d$. Then the affine maps $T_1,\ldots, T_N \colon \R^d \to \R^d$ defined by $T_ix:=A_ix+v_i$ preserve an affine subspace of $\R^d$ with dimension $\ell$ if and only if there exists an $\ell$-dimensional subspace $W \subseteq \R^d$ such that for every $i=1,\ldots,N$ we have both  $A_iW\subseteq W$ and $(I-A_i)^{-1}v_i - (I-A_N)^{-1}v_N \in W$.
\end{lemma}
\begin{proof}
Suppose first that there exists an $\ell$-dimensional subspace $W$  of $\R^d$  such that for every $i=1,\ldots,N$ we have both  $A_iW\subseteq W$ and $(I-A_i)^{-1}v_i - (I-A_N)^{-1}v_N \in W$. Let $X \subseteq \R^d$ denote the affine subspace $X:=W+(I-A_N)^{-1}v_N\subseteq \R^d$ which clearly has dimension $\ell$. We observe that $X=W+(I-A_i)^{-1}v_i$ for every $i=1,\ldots,N$ since for every $i$ the difference $(I-A_i)^{-1}v_i-(I-A_N)^{-1}v_N$ belongs to $W$. We claim that $T_iX\subseteq X$ for every $i=1,\ldots,N$. Given $x \in X$ and $i \in\{1,\ldots,N\}$ let us write $x=w+(I-A_i)^{-1}v_i$ where $w \in W$. We have
\begin{align*}T_ix &:= A_ix + v_i\\
&=A_iw + A_i(I-A_i)^{-1}v_i + v_i\\
&=A_iw + A_i(I-A_i)^{-1}v_i + (I-A_i)(I-A_i)^{-1}v_i\\
&=A_iw +(I-A_i)^{-1}v_i
\end{align*}
and since $A_iw \in A_iW\subseteq W$ this is an element of $X=W+(I-A_i)^{-1}v_i$ as required.

In the other direction, suppose that $X$ is an $\ell$-dimensional affine subspace of $\R^d$ such that $T_iX \subseteq X$ for every $i=1,\ldots,N$. Since each $T_i$ is contracting with respect to some fixed norm on $\R^d$ it has a unique fixed point in $\R^d$, and this fixed point is precisely $(I-A_i)^{-1}v_i$. On the other hand each $T_i$ is also a contraction of $X$ and hence has a unique fixed point in $X$, and these observations combined imply that $(I-A_i)^{-1}v_i \in X$ for every $i=1,\ldots,N$. Let us write $X=W+(I-A_N)^{-1}v_N$ where $W$ is a vector subspace of $\R^d$ with the same dimension as $X$. Since $(I-A_i)^{-1}v_i \in X$ for every $i=1,\ldots,N$ it follows that $(I-A_i)^{-1}v_i-(I-A_N)^{-1}v_N \in W$ for every $i=1,\ldots,N-1$ and clearly this also holds for $i=N$. We claim that $A_iW \subseteq W$ for every $i=1,\ldots,N$. Given $i \in \{1,\ldots,N\}$ and $w \in W$, write $w=x-(I-A_i)^{-1}v_i$ where $x \in X$. We have
\begin{align*}A_iw &= A_i(x-(I-A_i)^{-1}v_i)\\
&=A_ix - A_i(I-A_i)^{-1}v_i\\
&=A_ix-(I-A_i)^{-1}v_i + (I-A_i)^{-1}v_i - A_i(I-A_i)^{-1}v_i\\
&=A_ix-(I-A_i)^{-1}v_i+v_i\\
&=T_ix-(I-A_i)^{-1}v_i\end{align*}
and this belongs to $W=X-(I-A_i)^{-1}v_i$ since $T_ix \in X$. Thus $A_iW\subseteq W$ for every $i=1,\ldots,N$ as required.
 \end{proof}
\begin{lemma}
Let $N, \ell, d$ be integers satisfying $1 \leq N-1 \leq \ell <d$ and let $A_1,\ldots,A_N \in M_d(\R)$ be linear maps none of which has $1$ as an eigenvalue. Then the following are equivalent:
\begin{enumerate}[(i)]
\item
For every $v_1,\ldots,v_N \in \R^d$  there exists a subspace $W \subseteq \R^d$ of dimension at most $\ell$ such that for every $i=1,\ldots,N$ we have both $A_iW\subseteq W$ and $(I-A_i)^{-1}v_i-(I-A_N)^{-1}v_N \in W$.
\item
Every $(N-1)$-dimensional subspace $U \subseteq \R^d$  is contained in a subspace $W \subseteq \R^d$ of dimension at most $\ell$ such that $AW \subseteq W$ for every $A$ in the subalgebra of $M_d(\R)$ generated by $A_1,\ldots,A_N$.
\end{enumerate}
\end{lemma}
\begin{proof}
Supposing that (i) holds, let $U \subseteq \R^d$ be an $(N-1)$-dimensional subspace and $u_1,\ldots,u_{N-1}$ a basis for $U$. Define $v_i:=(I-A_i)u_i$ for $i=1,\ldots,N-1$ and define also $v_N:=0$. Applying (i), let $W \subseteq \R^d$ be a subspace of dimension at most $\ell$ such that for every $i=1,\ldots,N$ we have both $A_iW\subseteq W$ and $(I-A_i)^{-1}v_i-(I-A_N)^{-1}v_N \in W$. We have $u_i \in W$ for every $i=1,\ldots,N-1$ and therefore $U \subseteq W$. Clearly the set of all $A \in M_d(\R)$ which satisfy $AW \subseteq W$ is an algebra, and this set also contains $A_1,\ldots,A_N$, so (ii) holds. We have proved (i)$\implies$(ii).

If (ii) holds, given $v_1,\ldots,v_N \in \R^d$ let $U \subseteq \R^d$ be an $(N-1)$-dimensional subspace which contains the span of the vectors $(I-A_i)^{-1}v_i-(I-A_N)^{-1}v_N$ for $i=1,\ldots,N$. By (ii) there exists a subspace $W \subseteq \R^d$ of dimension at most $\ell$ such that $U\subseteq W$ and such that $AW \subseteq W$ for every $A$ in the subalgebra of $M_d(\R)$ generated by $A_1,\ldots,A_N$. Clearly $W$ has the properties required by (i). This proves (ii)$\implies$(i).\end{proof}
The remaining two lemmas treat Theorem \ref{th:first-main}(i):
\begin{lemma}
Let $N, \ell, d$ be integers satisfying $1 \leq N-1 \leq \ell <d$ and let $A_1,\ldots,A_N \in M_d(\R)$ be linear maps all of which are contracting with respect to some norm on $\R^d$. Let $\mathcal{Z}_\ell\subseteq \R^d \oplus \cdots \oplus \R^d\simeq \R^{dN}$ denote the set of all $(v_1,\ldots,v_N)$ with the property that the affine maps $T_1,\ldots, T_N \colon \R^d \to \R^d$ defined by $T_ix:=A_ix+v_i$ preserve an affine subspace of $\R^d$ with dimension at most $\ell$. Then $\mathcal{Z}_\ell$ is a subvariety of $\R^{dN}$.
\end{lemma}
\begin{proof}
By Lemma \ref{le:re} we have $(v_1,\ldots,v_N) \in \mathcal{Z}_\ell$ if and only if there exists a subspace $W \subseteq \R^d$ of dimension at most $\ell$ such that for every $i=1,\ldots,N$ we have both  $A_iW\subseteq W$ and $(I-A_i)^{-1}v_i - (I-A_N)^{-1}v_N \in W$. Let $H \subset M_d(\R)$ denote the semigroup generated by $A_1,\ldots,A_N$ together with the identity and define
\[\widehat{W}:=\Span \left\{B\left(I-A_{i_1})^{-1}v_{i} - (I-A_N)^{-1}v_N\right)\colon B \in H\text{ and }1 \leq i \leq N-1\right\}.\]
We have  $A_i\widehat{W}\subseteq \widehat{W}$ and $(I-A_i)^{-1}v_i - (I-A_N)^{-1}v_N \in \widehat{W}$ for every $i=1,\ldots,N$ and it is clear that any subspace $W \subseteq \R^d$ with those properties must contain $\widehat{W}$. We conclude from this that $(v_1,\ldots,v_N) \in \mathcal{Z}_\ell$ if and only if $\dim \widehat{W}\leq \ell$. The latter holds if and only if for every choice of $B_1,\ldots,B_{\ell+1} \in H$ and $i_1,\ldots,i_{\ell+1} \in \{1,\ldots,N-1\}$ the vectors
\[B_j \left((I-A_{i_j})^{-1}v_{i_j} - (I-A_N)^{-1}v_N\right)\]
for $j=1,\ldots,\ell+1$ are linearly dependent. This in turn holds if and only if  for every choice of $B_1,\ldots,B_{\ell+1} \in H$ and $i_1,\ldots,i_{\ell+1} \in \{1,\ldots,N-1\}$ and every linear map $P \colon \R^d \to \R^{\ell+1}$ the vectors 
\[PB_j \left((I-A_{i_j})^{-1}v_{i_j} - (I-A_N)^{-1}v_N\right)\]
for $j=1,\ldots,\ell+1$ are linearly dependent. The desired result will therefore follow if for every \emph{fixed} choice of $B_1,\ldots,B_{\ell+1} \in H$ and $i_1,\ldots,i_{\ell+1} \in \{1,\ldots,N-1\}$ and every \emph{fixed} linear map $P \colon \R^d \to \R^{\ell+1}$, the set of all $(v_1,\ldots,v_N) \in \R^{dN}$ such that the vectors 
\[PB_j \left((I-A_{i_j})^{-1}v_{i_j} - (I-A_N)^{-1}v_N\right)\]
for $j=1,\ldots,\ell+1$ are linearly dependent is a subvariety. (This is true since $\mathcal{Z}_\ell$ is precisely equal to the intersection of all subvarieties of the form just described, and it is clear from the definition that the intersection of an arbitrary family of subvarieties of $\R^{dN}$ is itself a subvariety.) But those $\ell+1$ vectors are linearly dependent if and only if they form the columns of an $(\ell+1)\times (\ell+1)$ matrix with determinant zero, and this is clearly equivalent to the statement that $(v_1,\ldots,v_N)$ belongs to the zero locus of a certain polynomial $\R^{dN} \to \R$ depending on $P$, $B_1,\ldots,B_{\ell+1}$ and $i_1,\ldots,i_{\ell+1}$. The result is proved.
\end{proof}
To complete Theorem \ref{th:first-main}(i) we note the following standard result:
\begin{lemma}
Let $m \geq 1$ and let $\mathcal{Z}$ be a subvariety of $\R^m$. Then either $\mathcal{Z}=\R^m$, or $\mathcal{Z}$ has zero Lebesgue measure.
\end{lemma}
\begin{proof}
To prove the lemma it suffices to show that for every nonzero polynomial $p \colon \R^m \to \R$ the set $\{x \in \R^m \colon p(x)=0\}$ has zero Lebesgue measure, for which we use induction on $m$. The case $m=1$ is trivial since a nonzero polynomial $\R \to \R$ can have only finitely many zeros. Given the case $m$, consider a nonzero polynomial $p \colon \R^{m+1} \to \R$ which we write in the form $p(x_1,\ldots,x_{m+1})=\sum_{j=0}^n x_{m+1}^j p_j(x_1,\ldots,x_m)$ for suitable polynomials $p_j \colon \R^m \to \R$. Without loss of generality we may assume that $p_n$ is not the zero polynomial. By the induction hypothesis, for Lebesgue almost every $(x_1,\ldots,x_m) \in \R^m$ we have $p_n(x_1,\ldots,x_m)\neq 0$. For these values of $(x_1,\ldots,x_m)$ the polynomial $x_{m+1} \mapsto \sum_{j=0}^n x_{m+1}^j p_j(x_1,\ldots,x_m)$ is not the zero polynomial, hence evaluates to zero for only finitely many values of $x_{m+1} \in \R$. Thus for Lebesgue almost every $(x_1,\ldots,x_m) \in \R^m$ the set $\{x_{m+1} \in \R \colon p(x_1,\ldots,x_{m+1})=0\}$ has zero Lebesgue measure. It follows by Tonelli's theorem that $\{x \in \R^{m+1} \colon p(x)=0\}$ has zero Lebesgue measure. This completes the induction step and proves the lemma.
\end{proof}

\subsection{Proof of Theorem \ref{th:zeroth-main}}
Let $(T_1,\ldots,T_N)$ be an affine iterated function system acting on $\R^d$ and for every $i=1,\ldots, N$ let $A_i$ denote the linearisation of $T_i$. Let $\mathcal{A}$ denote the smallest subalgebra of $M_d(\R)$ which contains $A_1,\ldots,A_N$ and the identity, and let $D$ denote the dimension of $\mathcal{A}$ as a vector space. If $(N-1)D \geq d$ then the result is trivial since we may take $X:=\R^d$, so we assume otherwise.  We will show that the property described in Theorem \ref{th:first-main}(ii) holds with $\ell:=(N-1)D<d$. If $U \subseteq \R^d$ is a vector space of dimension $(N-1)$, let $u_1,\ldots,u_{N-1}$ be a basis for $U$. Define
\[W:=\left\{Au \colon A \in \mathcal{A}\text{ and }u \in U\right\}.\]
It is clear that $W$ is a vector space such that $AW\subseteq W$ for every $A$ in the subalgebra of $M_d(\R)$ generated by $A_1,A_2,\ldots,A_N$, and that $U \subseteq W$ since $I \in \mathcal{A}$. Moreover we have
\[W=\Span \bigcup_{i=1}^{N-1}\left\{Au_i \colon A \in \mathcal{A}\right\}.\]
This is the span of a union of $(N-1)$ subspaces of $\R^d$ each of which is a linear image of $\mathcal{A}$ and hence each of which has dimension at most $D$. Consequently $\dim W \leq (N-1)D$. It follows by  Theorem \ref{th:first-main}(ii) that $(T_1,\ldots,T_N)$ admits an invariant affine subspace $X \subseteq \R^d$ of dimension not greater than $(N-1)D$. It follows from the uniqueness clause of Hutchinson's theorem that any closed, nonempty subset of $\R^d$ which is mapped inside itself by every $T_i$ must contain the attractor of $(T_1,\ldots,T_N)$ as a subset. In particular the attractor of $(T_1,\ldots,T_N)$ is a subset of $X$. The Hausdorff dimension of that attractor is therefore bounded by the dimension of $X$, which in turn is bounded by $(N-1)D$.

\section{Acknowledgements}
This research was partially supported by the Leverhulme Trust (Research Project Grant RPG-2016-194). The author thanks C. Sert, T. M. Jordan, P. Shmerkin and B. Solomyak for helpful conversations and remarks. The author also thanks two anonymous referees for suggesting a number of improvements to the exposition.

\bibliographystyle{acm}
\bibliography{affine-affine}

\end{document}